\documentclass{ws}

\newtheorem{thm}{Theorem}[section]
\newtheorem{lem}[thm]{Lemma}
\newtheorem{prop}[thm]{Proposition}
\newtheorem{cor}[thm]{Corollary}
\newtheorem{defn}{Definition}[section]
\newtheorem{rem}{Definition}[section]

\newfont{\got}{eufm9 scaled 1095}
\def\X{\mbox{\got X}}

\def\X{\mbox{\got X}}
\def\s{\mbox{\got S}}

\newcommand{\F}{\mathcal{F}}
\newcommand{\B}{\mathcal{B}}
\newcommand{\W}{\mathcal{W}}
\newcommand{\K}{\mathcal{K}}

\newcommand{\D}{{\rm d}}
\newcommand{\tr}{{\rm tr}}
\newcommand{\pd}{\partial}
\newcommand{\ddx}{\frac{\pd}{\pd x^i}}
\newcommand{\ddy}{\frac{\pd}{\pd y^i}}
\newcommand{\ddu}{\frac{\pd}{\pd u^i}}
\newcommand{\ddv}{\frac{\pd}{\pd v^i}}

\newfont{\w}{msbm9 scaled\magstep1}
\def\R{\mbox{\w R}}

\def\Z{\mbox{\w Z}}

\newcommand{\thmref}[1]{Theorem~\ref{#1}}

\newcommand{\lemref}[1]{Lemma~\ref{#1}}
\newcommand{\prpref}[1]{Proposition~\ref{#1}}
\newcommand{\corref}[1]{Corollary~\ref{#1}}
\newcommand{\eqref}[1]{(\ref{#1})}

\newcommand{\grad}[1]{{\rm grad}(#1)}
\begin{document}

\title{ON HYPERCOMPLEX PSEUDO-HERMITIAN MANIFOLDS}

\author{KOSTADIN GRIBACHEV}
\address{Faculty of Mathematics and Informatics, University of Plovdiv, \\
236 Bulgaria Blvd., Plovdiv 4003, Bulgaria\\
E-mail: costas@uni-plovdiv.bg}

\author{MANCHO MANEV}
\address{Faculty of Mathematics and Informatics, University of Plovdiv, \\
236 Bulgaria Blvd., Plovdiv 4003, Bulgaria\\
E-mail: mmanev@uni-plovdiv.bg}

\author{STANCHO DIMIEV}
\address{Institute of Mathematics and Informatics,
Bulgarian Academy of Sciences, \\
Acad. G. Bonchev Str., Bl. 8, Sofia 1113, Bulgaria, \\
E-mail: sdimiev@math.bas.bg}

\maketitle

%%%%%%%%%%%%%%%%%%%%%%%%%%%%%%%%%%%%%%%%%%%%%%%%%%%%%%%%a
\abstracts{The class of the hypercomplex pseudo-Hermitian
manifolds is considered. The flatness of the considered manifolds
with the 3 parallel complex structures is proved. Conformal
transformations of the metrics are introduced. The conformal
invariance and the conformal equivalence of the basic types
manifolds are studied. A known example is characterized in
relation to the obtained results. }

%\footnote{Key words and phrases: almost hypercomplex manifold, pseudo-Hermitian metric, indefinite metric.}

%\footnote{Mathematics Subject Classification (2000): Primary 53C15, 53C50; Secondary 32Q60, 32Q15. }

%%%%%%%%%%%%%%%%%%%%%%%%%%%%%%%%%%%%%%%%%%%%%%%%%%%%%%%%%%%%%%0
\section*{Introduction}

This paper is a continuation of the same authors's paper
\cite{GriMaDi} which is inspired by the seminal work \cite{AlMa}
of D.~V.~Alekseevsky and S.~Marchiafava. We follow a parallel
direction including skew-Hermitian metrics with respect to the
almost hypercomplex structure.

In the first section we give some necessary facts concerning the
almost hypercomplex pseudo-Hermitian manifolds introduced in
\cite{GriMaDi}.

In the second one we consider the special class of (integrable)
hypercomplex pseudo-Hermitian manifolds, namely
pseudo-hyper-K\"ahler manifolds. Here we expose the proof of the
mentioned in \cite{GriMaDi} statement that each
pseudo-hyper-K\"ahler manifold is flat.

The third section is fundamental for this work. A study of the
group of conformal transformations of the metric is initiated
here. The conformal invariant classes and the conformal equivalent
class to the class of the pseudo-hyper-K\"ahler manifolds are
found.

Finally, we characterize a known example in terms of the conformal
transformations.

%%%%%%%%%%%%%%%%%%%%%%%%%%%%%%%%%%%%%%%%%%%%%%%%%%%%%%%%%%%%%%%%1-1

\section{Preliminaries}

\subsection{Hypercomplex pseudo-Hermitian structures in a real vector
space}

Let $V$ be a real $4n$-dimensional vector space. By $\left\{
\frac{\pd}{\pd x^i},\frac{\pd}{\pd y^i},\frac{\pd}{\pd
u^i},\frac{\pd}{\pd v^i} \right\}$, $i=1,2$,$\dots,n$, is denoted
a (local) basis on $V$. Each vector $x$ of $V$ is represented in
the mentioned basis as follows
   \begin{equation}\label{11}
x=x^i\ddx+y^i\ddy+u^i\ddu+v^i\ddv.
   \end{equation}

A standard complex structure on $V$ is defined as in \cite{So}:
\begin{equation}\label{12}
\begin{array}{llll}
J_1\ddx=\ddy, &J_1\ddy=-\ddx, &J_1\ddu=-\ddv, &J_1\ddv=\ddu;
\\[4pt]
J_2\ddx=\ddu, &J_2\ddy=\ddv, &J_2\ddu=-\ddx, &J_2\ddv=-\ddy;
\\[4pt]
J_3\ddx=-\ddv, &J_3\ddy=\ddu, &J_3\ddu=-\ddy, &J_3\ddv=\ddx.\\[4pt]
\end{array}
\end{equation}

The following properties about $J_\alpha$ are direct consequences
of \eqref{12}
\begin{equation}\label{13}
\begin{array} {l}
J_1^2=J_2^2=J_3^2=-Id,\\
J_1J_2=-J_2J_1=J_3,\quad J_2J_3=-J_3J_2=J_1,\quad
J_3J_1=-J_1J_3=J_2.
\end{array}
\end{equation}

If $x \in V$, i.e. $x(x^i,y^i,u^i,v^i)$ then according to
\eqref{12} and \eqref{13} we have
\begin{equation}\label{14}
J_1x(-y^i,x^i,v^i,-u^i),\quad J_2x(-u^i,-v^i,x^i,y^i),\quad
J_3x(v^i,-u^i,y^i,-x^i).
\end{equation}

\begin{defn}[\cite{AlMa}]\label{d1}
A triple $H=(J_1,J_2,J_3)$ of anticommuting complex structures on
$V$ with $J_3=J_1J_2$ is called \emph{a hypercomplex structure} on
$V$;
\end{defn}

A bilinear form $f$ on $V$ is defined as ordinary, $f:\; V \times
V \rightarrow\R$. We denote by $\B(V)$ the set of all bilinear
forms on $V$. Each $f$ is a tensor of type $(0,2)$, and $\B(V)$ is
a vector space of dimension $16n^2$.

Let $J$ be a given complex structure on $V$. A bilinear form $f$
on $V$ is called \emph{Hermitian} (respectively,
\emph{skew-Hermitian}) with respect to $J$ if the identity
$f(Jx,Jy)=f(x,y)$ (respectively, $f(Jx,Jy)=-f(x,y)$ holds true.

\begin{defn}[\cite{AlMa}]\label{d2}
A bilinear form $f$ on $V$ is called \emph{an Hermitian bilinear
form} with respect to  $H=(J_\alpha)$ if it is Hermitian with
respect to any complex structure $J_\alpha, \alpha=1,2,3$, i.e.
\begin{equation}\label{15}
f(J_\alpha x,J_\alpha y)=f(x,y)\qquad \forall\; x, y \in V.
\end{equation}
\end{defn}

We denote by $L_0=\B_H(V)$ the set of all Hermitian bilinear forms
on $V$. The  notion of pseudo-Hermitian bilinear forms is
introduced by the following
\begin{defn}[\cite{GriMaDi}]\label{d3}
A bilinear form $f$ on $V$ is called \emph{a pseudo-Hermitian
bilinear form with respect to} $H=(J_1,J_2,J_3)$, if it is
Hermitian with respect to $J_\alpha$ and skew-Hermitian with
respect to $J_\beta$ and $J_\gamma$, i.e.
\begin{equation}\label{16}
    f(J_\alpha x,J_\alpha y)=-f(J_\beta x,J_\beta y)=-f(J_\gamma x,J_\gamma y)=f(x,y)\quad \forall\; x, y \in
    V,
\end{equation}
where $(\alpha,\beta,\gamma)$ is a circular permutation of
$(1,2,3)$.
\end{defn}

We denote $f\in L_\alpha \subset \B(V)$ $(\alpha=0,1,2,3)$ when
$f$ satisfies the conditions \eqref{15} and \eqref{16},
respectively.

In \cite{AlMa} is introduced a pseudo-Euclidian metric $g$ with
signature $(2n,2n)$ as follows
\begin{equation}\label{115}
g(x,y):=\sum_{i=1}^n \left(-x^ia^i-y^ib^i+u^ic^i+v^id^i\right),
\end{equation}
where $x(x^i,y^i,u^i,v^i),\ y(a^i,b^i,c^i,d^i) \in V,\
i=1,2,\dots,n$. This metric satisfies the following properties
\begin{equation}\label{117}
g(J_1x,J_1y)=-g(J_2x,J_2y)=-g(J_3x,J_3y)=g(x,y).
\end{equation}
This means that the pseudo-Euclidean metric $g$ belongs to $L_1$.

The form $g_1: g_1(x,y)=g(J_1x,y)$ coincides with the K\"ahler
form $\Phi$ which is Hermitian with respect to $J_\alpha$, i.e.
\[
\Phi(J_\alpha x,J_\alpha y)=\Phi(x,y),\quad \alpha=1,2,3,\qquad
\Phi\in L_0.
\]

The attached to $g$ associated bilinear forms $g_2:
g_2(x,y)=g(J_2x,y)$ and $g_3: g_3(x,y)=g(J_3x,y)$ are symmetric
forms with the properties
\begin{equation}\label{120}
\begin{array}{l}
-g_2(J_1x,J_1y)=-g_2(J_2x,J_2y)=g_2(J_3x,J_3y)=g_2(x,y),\\
-g_3(J_1x,J_1y)=g_3(J_2x,J_2y)=-g_3(J_3x,J_3y)=g_3(x,y),
\end{array}
\end{equation}
i.e. $g_2 \in L_3$, $g_3 \in L_2$.

It follows that the K\"ahler form $\Phi$ is Hermitian regarding
$H$ and the metrics $g, g_2, g_3$ are pseudo-Hermitian  of
different types with signature $(2n,2n)$.

Now we recall the following notion:
\begin{defn}[\cite{GriMaDi}]\label{d14}
The structure $(H,G):=(J_1, J_2, J_3,g,\Phi,g_2,g_3)$ is called
\emph{a hypercomplex pseudo-Hermitian structure} on $V$.
\end{defn}

%%%%%%%%%%%%%%%%%%%%%%%%%%%%%%%%%%%%%%%%%%%%%%%%%%%%%%%%1-2

\subsection{Structural tensors on an almost $(H,G)$-manifold}

Let $(M,H)$ be an almost hypercomplex manifold \cite{AlMa}. We
suppose that $g$ is a symmetric tensor field of type $(0,2)$. If
it induces a pseudo-Hermitian inner product in $T_pM$, $p \in M$,
then $g$ is called \emph{a pseudo-Hermitian metric on $M$}. The
structure $(H,G):=(J_1,J_2,J_3,g,\Phi,g_2,g_3)$ is called \emph{an
almost hypercomplex pseudo-Hermitian structure on $M$} or in short
\emph{an almost $(H,G)$-structure on $M$}. The manifold $M$
equipped with $H$ and $G$, i.e. $(M,H,G)$, is called \emph{an
almost hypercomplex pseudo-Hermitian manifold}, or in short
\emph{an almost $(H,G)$-manifold}. \cite{GriMaDi}

The 3 tensors of type $(0,3)$ $F_\alpha: F_\alpha (x,y,z)=g\bigl(
\left( \nabla_x J_\alpha \right)y,z\bigr), \alpha=1,2,3$, where
$\nabla$ is the Levi-Civita connection generated by $g$, is called
\emph{structural tensors of the almost} $(H,G)$-\emph{manifold}.
\cite{GriMaDi}

The structural tensors satisfy the following
properties:
\begin{equation}\label{32}
\begin{array}{l}
    F_1(x,y,z)=F_2(x,J_3y,z)+F_3(x,y,J_2z),\\
    F_2(x,y,z)=F_3(x,J_1y,z)+F_1(x,y,J_3z),\\
    F_3(x,y,z)=F_1(x,J_2y,z)-F_2(x,y,J_1z);
\end{array}
\end{equation}
\begin{equation}\label{34}
\begin{array}{l}
    F_1(x,y,z)=-F_1(x,z,y)=-F_1(x,J_1y,J_1z),\\
    F_2(x,y,z)=F_2(x,z,y)=F_2(x,J_2y,J_2z),\\
    F_3(x,y,z)=F_3(x,z,y)=F_3(x,J_3y,J_3z).
\end{array}
\end{equation}

Let us recall the Nijenhuis tensors $N_\alpha(X,Y)
=\frac{1}{2}\left[\left[J_\alpha,J_\alpha\right]\right](X,Y)$ for
almost complex structures $J_\alpha$ and $X,Y \in \X(M)$, where
\[
\left[\left[J_\alpha,J_\alpha\right]\right](X,Y)=
2\bigl\{\left[J_\alpha X,J_\alpha Y \right]
    -J_\alpha\left[J_\alpha X,Y \right]
    -J_\alpha\left[X,J_\alpha Y \right]
    -\left[X,Y \right]\bigr\}.
\]

It is well known that the almost hypercomplex structure
$H=(J_\alpha)$ is a hypercomplex structure if
$\left[\left[J_\alpha,J_\alpha\right]\right]$ vanishes for each
$\alpha=1,2,3$. Moreover it is known that one almost hypercomplex
structure $H$ is hypercomplex if and only if two of the structures
$J_\alpha$ $(\alpha=1,2,3)$ are integrable. This means that two of
the tensors $N_\alpha$ vanish. \cite{AlMa}

We recall also the following definitions. Since $g$ is Hermitian
metric with respect to $J_1$, according to \cite{GrHe} the class
$\W_4$ is a subclass of the class of Hermitian manifolds. If
$(H,G)$-manifold belongs to $\W_4$, with respect to $J_1$, then
the almost complex structure $J_1$ is integrable and
\begin{equation}\label{35}
\begin{array}{l}
F_1(x,y,z)=\frac{1}{2(2n-1)} \left[
g(x,y)\theta_1(z)-g(x,z)\theta_1(y)\right.\\
\qquad\qquad\qquad\qquad\quad\left.-g(x,J_1y)\theta_1(J_1z)+g(x,J_1z)\theta_1(J_1y)
\right],
\end{array}
\end{equation}
where $\theta_1(\cdot)=g^{ij}F_1(e_i,e_j,\cdot)=\delta\Phi(\cdot)$
for the basis $\{e_i\}_{i=1}^{4n}$, and $\delta$ -- the
coderivative.

On other side the metric $g$ is a skew-Hermitian with respect to
$J_2$ and $J_3$, i.e. $g(J_2x,J_2y)=g(J_3x,J_3y)=-g(x,y)$. A
classification of all almost complex manifolds with skew-Hermitian
metric (Norden metric or B-metric) is given in \cite{GaBo}. One of
the basic classes of integrable almost complex manifolds
 with skew-Hermitian metric is \(\W_1\). It is known that if an almost $(H,G)$-manifold belongs to
 \(\W_1(J_\alpha), \alpha=2,3,\) then $J_\alpha$ is integrable and the following equality holds
\begin{equation}\label{36}
\begin{array}{l}
F_\alpha(x,y,z)=\frac{1}{4n} \left[
g(x,y)\theta_\alpha(z)+g(x,z)\theta_\alpha(y)\right.\\
\qquad\qquad\qquad\quad\left.   +g(x,J_\alpha
y)\theta_\alpha(J_\alpha z) +g(x,J_\alpha z)\theta_\alpha(J_\alpha
y)\right],
\end{array}
\end{equation}
where \(\theta_\alpha(z)=g^{ij}F_\alpha (e_i,e_j,z)\), \(\alpha
=2,3\), for an arbitrary basis \(\{e_i\}_{i=1}^{4n}\).

When \eqref{35} is satisfied for $(M,H,G)$, we say that $(M,H,G)
\in \W(J_1)$. In the case, $(M,H,G)$ satisfies \eqref{36} for
$\alpha =2$ or $\alpha =3$, we say $(M,H,G) \in \W(J_2)$ or
$(M,H,G) \in \W(J_3)$. Let us denote the class
$\W:=\bigcap_{\alpha=1}^3\W(J_\alpha)$.

The next theorem gives a sufficient condition an almost
$(H,G)$-manifold to be integrable.
\begin{thm}[\cite{GriMaDi}] \label{t31}
Let $(M,H,G)$ belongs to the class $ \W(J_\alpha) \bigcap
\W(J_\beta)$. Then $(M,H,G)$ is of class $\W(J_\gamma)$ for all
cyclic permutations $(\alpha, \beta, \gamma)$ of $(1,2,3)$.
\end{thm}

Let us remark that necessary and sufficient conditions $(M,H,G)$
to be in $\W$ are
\begin{equation}\label{37}
\theta_\alpha\circ J_\alpha =-\frac{2n}{2n-1}\theta_1\circ
J_1,\qquad \alpha = 2,3.
\end{equation}

%%%%%%%%%%%%%%%%%%%%%%%%%%%%%%%%%%%%%%%%%%%%%%%%%%%%%%%%%%%%%%%2

\section{Pseudo-hyper-K\"ahler manifolds}
\begin{defn}[\cite{GriMaDi}]
A pseudo-Hermitian manifold is called a
\emph{pseudo-hyper-K\"ahler manifold}, if $\nabla J_\alpha=0\;
(\alpha=1,2,3)$ with respect to the Levi-Civita connection
generated by $g$.
\end{defn}

It is clear, then $F_\alpha=0\, (\alpha=1,2,3)$ holds or the
manifold is K\"ahlerian with respect to $J_\alpha$, i.e. $(M,H,G)
\in \K(J_\alpha)$.

Immediately we obtain that if $(M,H,G)$ belongs to
$\K(J_\alpha)\bigcap \W(J_\beta)$ then $(M,H,G)\in\K(J_\gamma)$
for all cyclic permutations $(\alpha, \beta, \gamma)$ of
$(1,2,3)$.

Then the following sufficient condition for a $\K$-manifold is
valid.
\begin{thm}[\cite{GriMaDi}]\label{t33}
If $(M,H,G) \in \K(J_\alpha)\bigcap \W(J_\beta)$ then $M$ is a
pseudo-hyper-K\"ahler manifold $(\alpha\neq\beta \in \{1,2,3\})$.
\end{thm}

Let $(M^{4n},H,G)$ be a pseudo-hyper-K\"ahler manifold and
$\nabla$ be the Levi-Civita connection generated by $g$. The
curvature tensor seems as follows
\begin{equation}\label{41}
R(X,Y)Z=\nabla_X \nabla_Y Z - \nabla_Y \nabla_X Z -
\nabla_{\left[X,Y\right]} Z,
\end{equation}
and the corresponding tensor of type $(0,4)$ is
\begin{equation}\label{42}
R(X,Y,Z,W)=g\left(R(X,Y)Z,W\right),\quad \forall\; X,Y,Z,W \in
\X(M).
\end{equation}

\begin{lem}\label{l41}
The curvature tensor of a pseudo-hyper-K\"ahler manifold has the
following properties:
\begin{equation}\label{43}
\begin{array}{l}
R(X,Y,Z,W)=R(X,Y,J_1Z,J_1W)=R(J_1X,J_1Y,Z,W)\\
\qquad\qquad\qquad\; =-R(X,Y,J_2Z,J_2W)=-R(J_2X,J_2Y,Z,W)\\
\qquad\qquad\qquad\; =-R(X,Y,J_3Z,J_3W)=-R(J_3X,J_3Y,Z,W),
\end{array}
\end{equation}
\begin{equation}\label{44}
\begin{array}{l}
R(X,Y,Z,W)=R(X,J_1Y,J_1Z,W)\\
\qquad\qquad\qquad\;=-R(X,J_2Y,J_2Z,W)=-R(X,J_3Y,J_3Z,W).
\end{array}
\end{equation}
\end{lem}

\begin{proof}
The equality \eqref{43} is valid, because of \eqref{41},
\eqref{42}, the condition $\nabla J_\alpha=0$ $(\alpha=1,2,3)$,
the equality \eqref{117} and the properties of the curvature
$(0,4)$-tensor.

To prove \eqref{44}, we  will show at first that the property
$R(X,J_2Y,J_2Z,W)=-R(X,Y,Z,W)$ holds. Indeed, from \eqref{43} we
get
\[
R(J_2X,Y,Z,W)=R(X,J_2Y,Z,W),\quad R(X,Y,J_2Z,W)=R(X,Y,Z,J_2W)
\]
and $\s_{X,Y,Z} R(X,Y,J_2Z,J_2W)=0$, where $\s_{X,Y,Z}$ denotes
the cyclic sum regarding $X, Y, Z$. In the last equality we
replace $Y$ by $J_2Y$ and $W$ by $J_2W$. We get
\begin{equation}\label{i}
-R(X,J_2Y,J_2Z,W)-R(J_2Y,Z,J_2X,W)+R(Z,X,Y,W)=0.
\end{equation}
Replacing $Y$ by $Z$, and inversely, we get
\begin{equation}\label{ii}
-R(X,J_2Z,J_2Y,W)-R(J_2Z,Y,J_2X,W)+R(Y,X,Z,W)=0.
\end{equation}
As we have
\[
-R(J_2Z,Y,J_2X,W)=-R(Z,J_2Y,J_2X,W)=R(J_2Y,Z,X,W),
\]
with the help of \eqref{i} and \eqref{ii} we obtain

\begin{equation}\label{iii}
\begin{array}{l}
-R(X,J_2Y,J_2Z,W)-R(X,J_2Z,J_2Y,W)\\
+R(Z,X,Y,W)+R(Y,X,Z,W)=0.
\end{array}
\end{equation}

According to the first Bianchi identity and \eqref{43}, we obtain
\[
\begin{array}{l}
-R(X,J_2Z,J_2Y,W)=R(J_2Z,J_2Y,X,W)+R(J_2Y,X,J_2Z,W)\\
\phantom{-R(X,J_2Z,J_2Y,W)}=-R(Z,Y,X,W)-R(X,J_2Y,J_2Z,W).
\end{array}
\]
Then the equality \eqref{iii} seem as follows
\[
-2R(X,J_2Y,J_2Z,W)+R(Z,X,Y,W)-R(X,Y,Z,W)+R(Y,Z,X,W)=0
\]
By the first Bianchi identity  the equality is transformed in the
following
\[
-2R(X,J_2Y,J_2Z,W)-2R(X,Y,Z,W)=0,
\]
which is equivalent to
\begin{equation}\label{a}
R(X,J_2Y,J_2Z,W)=-R(X,Y,Z,W).
\end{equation}
As the tensor $R$ has the same properties with respect to $J_3$,
and to $J_2$, it follows that the next equality holds, too.
\begin{equation}\label{b}
R(X,J_3Y,J_3Z,W)=-R(X,Y,Z,W).
\end{equation}
Using \eqref{a} and \eqref{b} for $J_1=J_2J_3$ we get successively
that
\[
\begin{array}{l}
R(X,Y,Z,W)=R(X,J_1Y,J_1Z,W)\\
\qquad\qquad\quad\quad\;\,=R(X,J_2(J_3Y),J_2(J_3Z),W)
=-R(X,J_3Y,J_3Z,W),
\end{array}
\]
which completes the proof of \eqref{44}.
\end{proof}

Now we will prove a theorem which gives us a geometric
characteristic of the pseudo-hyper-K\"ahler manifolds.
\begin{thm}\label{R=0}
Each pseudo-hyper-K\"ahler manifold is a flat pseudo-Riemann\-ian
manifold with signature $(2n,2n)$.
\end{thm}

\begin{proof}
\lemref{l41} implies the properties
\begin{equation}\label{45}
\begin{array}{l}
-R(X,Y,Z,W)=R(X,J_1Y,Z,J_1W)\\
\qquad\qquad\qquad\quad\; =R(X,J_2Y,Z,J_2W)=R(X,J_3Y,Z,J_3W).
\end{array}
\end{equation}
As $J_1=J_2J_3$, we also have  the following
\[
\begin{array}{l}
R(X,J_1Y,Z,J_1W)=R(X,J_2(J_3Y),Z,J_2(J_3W))\\
\qquad\qquad\qquad\qquad\;\,
        =-R(X,J_3Y,Z,J_3W)=R(X,Y,Z,W).
\end{array}
\]
Comparing \eqref{45} with the last equality we receive
\[
-R(X,Y,Z,W)=R(X,J_1Y,Z,J_1W)=R(X,Y,Z,W),
\]
or $R\equiv 0$.
\end{proof}

%%%%%%%%%%%%%%%%%%%%%%%%%%%%%%%%%%%%%%%%%%%%%%%%%%%3
\section{Conformal transformations of the pseudo-Hermitian
metric}

The usual conformal transformation $c: \bar g = e^{2u}g$, where
$u$ is a differential function on $M^{4n}$, is known. Since
$g_\alpha(\cdot,\cdot)=g(J_\alpha\cdot,\cdot)$, the conformal
transformation of $g$ causes the same changes of the
pseudo-Hermitian metrics $g_2, g_3$ and the K\"ahler form
$\Phi\equiv g_1$. Then we say that it is given a conformal
transformation $c$ of $G$ to $\bar{G}$ determined by $u\in\F (M)$.
These conformal transformations form a group denoted by $C$. The
hypercomplex pseudo-Hermitian manifolds $(M,H,G)$ and
$(M,H,\bar{G})$ we call $C$-equivalent manifolds or
conformal-equivalent manifolds.

Let $\nabla$ and $\bar\nabla$ be the Levi-Civita connections
determined by the metrics $g$ and $\bar{g}$, respectively. The
known  condition for a Levi-Civita connection implies the
following relation
\begin{equation}\label{nab}
\bar\nabla_X Y=\nabla_X Y+\D{u}(X)Y+\D{u}(Y)X-g(X,Y)\grad{u}.
\end{equation}
Using \eqref{nab} and the definitions of structural tensors for
$\nabla$ and $\bar\nabla$ we obtain
\begin{equation}\label{F1}
\begin{array}{l}
\bar{F}_1(X,Y,Z)=e^{2u}\left[F_1(X,Y,Z)
-g(X,Y)\D{u}(J_1Z)+g(X,Z)\D{u}(J_1Y)\right.\\
\phantom{\bar{F}_1(X,Y,Z)=e^{2u}\left[F_1(X,Y,Z)\right.}
\left.+g(J_1X,Y)\D{u}(Z)-g(J_1X,Z)\D{u}(Y)\right],
\end{array}
\end{equation}
\begin{equation}\label{Fa}
\begin{array}{l}
\bar{F}_\alpha(X,Y,Z)=e^{2u}\left[F_\alpha(X,Y,Z)
+g(X,Y)\D{u}(J_\alpha Z)+g(X,Z)\D{u}(J_\alpha Y)\right.\\
\phantom{\bar{F}_\alpha(X,Y,Z)=e^{2u}\left[F_\alpha(X,Y,Z)\right.}
\left.-g(J_\alpha
X,Y)\D{u}(Z)-g(J_\alpha X,Z)\D{u}(Y)\right]
\end{array}
\end{equation}
for $\alpha=2,3$. The last two equalities imply the following
relations for the corresponding structural 1-forms
\begin{equation}\label{theta}
\bar\theta_1=\theta_1-2(2n-1)\D{u}\circ J_1,\qquad
\bar\theta_\alpha=\theta_\alpha+4n\D{u}\circ J_\alpha, \quad
\alpha=2,3.
\end{equation}

Let us denote the following (0,3)-tensors.
\begin{equation}\label{p1}
\begin{array}{l}
P_1(x,y,z)=F_1(x,y,z)\\
\phantom{P_1(x,y,z)=}-\frac{1}{2(2n-1)} \left[
g(x,y)\theta_1(z)-g(x,z)\theta_1(y)\right.\\
\qquad\qquad\qquad\qquad\qquad\left.-g(x,J_1y)\theta_1(J_1z)+g(x,J_1z)\theta_1(J_1y)
\right],
\end{array}
\end{equation}
\begin{equation}
\begin{array}{l}\label{Pa}
P_\alpha(x,y,z)=F_\alpha(x,y,z)\\
\phantom{P_\alpha(x,y,z)=}-\frac{1}{4n} \left[
g(x,y)\theta_\alpha(z)+g(x,z)\theta_\alpha(y)\right.\\
\qquad\qquad\qquad\qquad\left.   +g(x,J_\alpha
y)\theta_\alpha(J_\alpha z) +g(x,J_\alpha z)\theta_\alpha(J_\alpha
y)\right],\;\; \alpha=2,3.
\end{array}
\end{equation}
According to \eqref{35} and \eqref{36} it is clear that
\[
(M,H,G)\in\W(J_\alpha) \iff P_\alpha=0\quad (\alpha=1,2,3).
\]
The equalities \eqref{F1}--\eqref{theta} imply the following two
interconnections
\begin{equation}\label{PP}
\bar{P}_\alpha=e^{2u}P_\alpha,\quad\alpha=1,2,3;
\end{equation}
\begin{equation}\label{tt}
\bar{\theta}_\alpha\circ
J_\alpha+\frac{2n}{2n-1}\bar{\theta}_1\circ J_1=
\theta_\alpha\circ J_\alpha+\frac{2n}{2n-1}\theta_1 \circ
J_1,\quad\alpha=2,3.
\end{equation}

From \eqref{PP} we receive that each of %the classes
$\W(J_\alpha)\, (\alpha=1,2,3)$ is invariant with respect to the
conformal transformations of $C$, i.e. they are $C$-invariant
classes. Having in mind also \eqref{tt}, we state the validity of
the following
\begin{thm}\label{inv}
The class $\W$ of hypercomplex pseudo-Hermit\-ian manifolds is
$C$-invariant.
\end{thm}

Now we will determine the class of the (locally) $C$-equivalent
$\K$-manifolds. Let us denote the following subclass
$\W^0:=\left\{\W\;|\;d\left(\theta_1\circ J_1\right)=0\right\}$.

\begin{thm}\label{C-eq}
A hypercomplex pseudo-Hermitian manifold belongs to $\W^0$ if and
only if it is $C$-equivalent to a pseudo-hyper-K\"ahler manifold.
\end{thm}
\begin{proof}
Let $(M,H,G)$ be a pseudo-hyper-K\"ahler manifold, i.e.
$(M,H,G)\in\K$. Then $F_\alpha=\theta_\alpha=0$ ($\alpha=1,2,3$).
Hence \eqref{theta} has the form
\begin{equation}\label{the0}
\bar\theta_1=-2(2n-1)\D{u}\circ J_1,\qquad
\bar\theta_\alpha=4n\D{u}\circ J_\alpha,\quad \alpha=2,3.
\end{equation}
From \eqref{F1}, \eqref{Fa} and \eqref{the0} and having in mind
\eqref{35} and \eqref{36} we obtain that $(M,H,\bar{G})$ is a
$\W$-manifold. According to \eqref{the0} the 1-forms
$\bar\theta_\alpha\circ J_\alpha$ ($\alpha=1,2,3$) are closed.
Because of \eqref{37} the condition $\D(\bar\theta_1\circ J_1)=0$
is sufficient.

Conversely, let $(M,H,\bar{G})$ be a $\W$-manifold with closed
$\bar\theta_1\circ J_1$. Because of \eqref{37} the 1-forms
$\bar\theta_\alpha\circ J_\alpha$ ($\alpha=2,3$) are closed, too.
We determine the function $u$ as a solution of the differential
equation $\D{u}=-\frac{1}{2(2n-1)}\bar\theta_1\circ J_1$. Then by
an immediate verification we state that the transformation
$c^{-1}: g=e^{-2u}\bar{g}$ converts $(M,H,\bar{G})$ into
$(M,H,G)\in\K$. This completes the proof.
\end{proof}

Let us remark the following inclusions
\[
\K\subset\W^0\subset\W\subset\W(J_\alpha),\quad \alpha=1,2,3.
\]

Let $R, \rho, \tau$ and $\bar{R}, \bar\rho, \bar\tau$ be the
curvature tensors, the Ricci tensors, the scalar curvatures
corresponding to $\nabla$ and $\bar\nabla$, respectively. The
following tensor is curvature-like, i.e. it has the same
properties as $R$.
\[
\begin{array}{l}
\psi_1(S)(X,Y,Z,U)=g(Y,Z)S(X,U)-g(X,Z)S(Y,U)\\
\qquad\qquad\qquad\qquad\;+g(X,U)S(Y,Z)-g(Y,U)S(X,Z),
\end{array}
\]
where $S$ is a symmetric tensor.

Having in mind \eqref{nab} and \eqref{41}, we obtain
\begin{prop}\label{p-RbarR}
The following relations hold for the $C$-equivalent
$(H,G)$-mani\-folds
\begin{equation}\label{RbarR}
\begin{array}{c}
\bar{R}=e^{2u}\{R-\psi_1(S)\},\\[4pt]
\bar{\rho}=\rho-\tr{S}g-2(2n-1)S,\qquad
\bar\tau=e^{-2u}\{\tau-2(4n-1)\tr{S}\},
\end{array}
\end{equation}
where
\begin{equation}\label{S}
S(Y,Z)=S(Z,Y)=\left(\nabla_Y\D{u}\right)Z+\D{u}(Y)\D{u}(Z)-\frac{1}{2}\D{u}(\grad{\D{u}})g(Y,Z).
\end{equation}
\end{prop}

If $(M,H,G)$ is a $C$-equivalent $\W$-manifold to a $\K$-manifold,
i.e. $(M,H,G)\in \W^0$, then \prpref{p-RbarR} implies
\begin{cor}\label{R0}
A $\W^0$-mani\-fold has the following curvature characteristic
\[
R=\frac{1}{2(2n-1)}\left\{\psi_1(\rho)-\frac{\tau}{4n-1}\pi_1\right\},
\]
where
$\pi_1(X,Y,Z,U)=\frac{1}{2}\psi_1(g)=g(Y,Z)g(X,U)-g(X,Z)g(Y,U)$.
\end{cor}

It is well known that the $C$-invariant tensor of each
pseudo-Riemannian manifold is the so-called Weil tensor $W$. From
\eqref{RbarR} we receive immediately
\begin{equation}\label{W}
\bar W=e^{2u}W, \qquad
W=R-\frac{1}{2(2n-1)}\left\{\psi_1(\rho)-\frac{\tau}{4n-1}\pi_1\right\}.
\end{equation}

Let us remark that the vanishing of the Weil tensor is a necessary
and sufficient condition a pseudo-Riemannian manifold to be
conformal equivalent to a flat manifold with dimension greater
than 3.

This is confirmed by the combining of \thmref{R=0}, \thmref{C-eq}
and \corref{R0}, i.e. $(M,H,G)\in\W^0$ iff $W=0$ on $(M,H,G)$.

%      ***
Since each conformal transformation determines uniquely a
symmetric tensor $S$ by \eqref{S} then it takes an interest in the
consideration $S$ as a bilinear form on $T_pM$ belonging to each
of the components $L_\alpha$, $(\alpha=0,1,2,3)$.

Let $S\in L_0$.  In view of \eqref{15} $\tr{S}=0$ holds and
according to \eqref{RbarR} we receive $\bar\tau=e^{-2u}\tau$ and
an invariant tensor $W_0=R-\frac{1}{2(2n-1)}\psi_1(\rho)$. When
$W_0$ vanishes on $(M,H,G)$ then the curvature tensor has the form
$R=\frac{1}{2(2n-1)}\psi_1(\rho)$.

In the cases when $S\in L_\alpha$ $(\alpha=1,2,3)$ we consider
$(M,H,G)$ as an $\W^0$-manifold. Then according to \thmref{R=0}
and \thmref{C-eq} we have $\bar{R}=0$ on the $C$-equivalent
$\K$-manifold of $(M,H,G)$.

Now let $S\in L_1$.  By reason of $g\in L_1$ we have a cause for
the consideration of the possibility $S=\lambda g$. Hence
$\lambda=\frac{\tr{S}}{4n}=\frac{\tau}{8n(4n-1)}$. Then having in
mind \eqref{RbarR} $R=\frac{\tau}{4n(4n-1)}\pi_1$ holds true. From
here it is clear that if $S\in L_1$ then $(M,H,G)$ is an Einstein
manifold.

Let us consider the case when $S\in L_2$. Then according to
\eqref{16} $\tr{S}$ vanishes, and from \eqref{RbarR} $\tau$
vanishes, too. Because of $g_3\in L_2$ we consider $S=\lambda
g_3$, whence $\lambda=-\frac{\tr{(S\circ J_3)}}{4n}$. Then
\eqref{RbarR} implies $R=\frac{\tr{(S\circ J_3)}}{4n}\pi_3^{J_3}$,
where $\pi_3^{J_3}$ is the following tensor $\pi_3$ with respect
to the complex structure $J=J_3$
\[
\pi_3(X,Y,Z,U)=-\pi_1(X,Y,JZ,U)-\pi_1(X,Y,Z,JU).
\]
It is known \cite{GaBo} that $\pi_3$ is a K\"ahler curvature-like
tensor, i.e. it satisfies the property
$\pi_3(X,Y,JZ,JU)=-\pi_3(X,Y,Z,U)$. Therefore in this case $R$ is
K\"ahlerian with respect to $J_3$ and the tensor $R^{*J_3}:
R^{*J_3}(X,Y,Z,U)=R(X,Y,Z,J_3U)$ is curvature-like. Then we obtain
immediately
\[
R=\frac{\tau(R^{*J_3})}{8n(2n-1)}\pi_3^{J_3},\qquad
\rho=-\frac{\tau(R^{*J_3})}{4n}g_3.
\]
Hence if $S\in L_2$ then $(M,H,G)$ is a *-Einstein manifold with
respect to $J_3$.

By an analogous way, in the case when $S\in L_3$ we receive that
$(M,H,G)$ is a *-Einstein manifold with respect to $J_2$.

%%%%%%%%%%%%%%%%%%%%%%%%%%%%%%%%%%%%%%%%%%%%%%%%%%%%%%%%%%%%%%%4

\section{A 4-dimensional pseudo-Riemannian spherical manifold with $(H,G)$-structure}

In \cite{GriMaDi} is considered a hypersurface $S_2^4$ in $\R_2^5$
by the equation
\begin{equation}\label{53}
-\left(z^1\right)^2-\left(z^2\right)^2
+\left(z^3\right)^2+\left(z^4\right)^2+\left(z^5\right)^2=1,
\end{equation}
where $Z\left(z^1,z^2,z^3,z^4,z^5\right)$ is the positional vector
of  $p\in S_2^4$.

Let $\left(u^1,u^2,u^3,u^4\right)$ be local coordinates
 of $p$ on $S_2^4$. The
hypersurface $S_2^4$ is defined by the scalar parametric
equations:
\begin{equation}\label{54}
\begin{array}{c}
z^1=\sinh u^1 \cos u^2,\quad z^2=\sinh u^1 \sin u^2,\quad z^3=\cosh u^1 \cos u^3\cos u^4, \\
z^4=\cosh u^1 \cos u^3\sin u^4,\quad z^5=\cosh u^1 \sin u^3.\\
\end{array}
\end{equation}
Further we consider the manifold on $\tilde S_2^4=S_2^4\backslash
\{(0,0,0,0,\pm 1)\}$, i.e. we omit two points for which $\{u^1\neq
0\}\cap \{u^3\neq (2k+1)\pi /2, k\in \Z\}$. The tangent space
$T_p\tilde{S}_2^4$ of $\tilde{S}_2^4$ in the point $p\in
\tilde{S}_2^4$ is determined by the vectors $ z_i=\frac{\pd Z}{\pd
u^i} (i=1,2,3,4)$. The vectors $z_i$ are linearly independent on
$\tilde S_2^4$, defined by \eqref{54}, and $T_p\tilde S_2^4$ has a
basis $(z_1,z_2,z_3,z_4)$ in every point $p\in \tilde{S}_2^4$.

The restriction of $\langle\cdot,\cdot\rangle$ from $\R_2^5$ to
$S_2^4$ is  a pseudo-Riemannian metric $g$ on $S_2^4$ with
signature $(2,2)$. The non-zero components $g_{ij}=\langle z_i,z_j
\rangle$ are
\begin{equation}\label{g}
g_{11}=-1,\; g_{22}=-\sinh ^2 u^1, \; g_{33}=\cosh ^2 u^1,\;
g_{44}=\cosh ^2 u^1\cos ^2 u^3.
\end{equation}

The hypersurface $S_2^4$ is equipped with an almost hypercomplex
structure $H=(J_\alpha ), (\alpha=1,2,3)$, where  the non-zero
components of the matrix of $J_\alpha$ with respect to the local
basis $\left\{\frac{\pd}{\pd u^i}\right\}_{i=1}^4$ are
\begin{equation}\label{J}
\begin{array}{ll}
 (J_1)_2^1=-\frac{1}{(J_1)_1^2}=-\sinh u^1,\;
 &
 (J_1)_4^3=-\frac{1}{(J_1)_3^4}=\cos u^3,\\
 (J_2)_3^1=-\frac{1}{(J_2)_1^3}=-\cosh u^1,\;
 &
 (J_2)_4^2=-\frac{1}{(J_2)_2^4}=-\coth u^1 \cos u^3,\\
 (J_3)_4^1=-\frac{1}{(J_3)_1^4}=\cosh u^1 \cos u^3,\;
 &
 (J_3)_2^3=-\frac{1}{(J_3)_3^2}=\tanh u^1.
\end{array}
\end{equation}

\begin{thm}[\cite{GriMaDi}]\label{t51}
The spherical pseudo-Riemannian 4-dimensional manifold, defined by
\eqref{54}, admits a hypercomplex
 pseudo-Hermitian structure on $\tilde{S}_2^4$, determined by \eqref{J} and \eqref{g},
 with respect to which it is of the class
$\W(J_1)$ but it does not belong to $\W$ and it has a constant
sectional curvature $k=1$.
\end{thm}

Let us consider a conformal transformation determined by the
function $u$ which is a solution of the equation $\D
u=-\frac{1}{2(2n-1)}(\theta_1\circ J_1)$, where the nonzero
component of $\theta_1$ with respect to the local basis
$\left\{\frac{\pd}{\pd u^i}\right\}$ $(i=1,2,3,4)$ is
$\theta_1\left(\frac{\partial}{\partial u^2}\right)=\frac{2\sinh^2
u^1}{\cosh u^1}$.

Since $\tilde{S}_2^4$ has a constant sectional curvature then the
Weil tensor is vanishes, i.e. $\tilde{S}_2^4$ is $C$-equivalent to
a flat $\K(J_1)$-manifold. If we admit that it is in $\K$, then
according to \thmref{C-eq} we obtain that the manifold
$(\tilde{S}_2^4,H,G)\in\W$ which is a contradiction. Therefore the
considered manifold is $C$-equivalent to a flat
$\K(J_1)$-manifold, but it is not a pseudo-hyper-K\"ahler
manifold. By direct verification we state that the tensor $S$ of
this conformal transformation belongs to $L_1$. Therefore
$(\tilde{S}_2^4,H,G)$ is an Einstein manifold.

\end{document}